\documentclass{amsart}

\makeatletter

\@namedef{subjclassname@2020}{%
  \textup{2020} Mathematics Subject Classification}
\makeatother 

\usepackage{mathrsfs}
\usepackage{amsthm,amssymb,amsfonts,latexsym,mathtools,thmtools}
\usepackage[T1]{fontenc}
\usepackage{tikz-cd} 
\usepackage{enumitem} 
\usepackage{amsmath}
\usepackage{hyperref} 
\usepackage{hyperref}
\hypersetup{
    colorlinks=true,
    linkcolor=blue,
    filecolor=blue,      
    urlcolor=blue,
    linktocpage=true
}

\newtheorem{theorem}{Theorem}[section]
\newtheorem{lemma}[theorem]{Lemma}
\newtheorem{proposition}[theorem]{Proposition}
\newtheorem{corollary}[theorem]{Corollary}

\theoremstyle{definition}
\newtheorem{definition}[theorem]{Definition}
\newtheorem{example}[theorem]{Example}

\newtheorem{remark}[theorem]{Remark}

\numberwithin{equation}{section}

\begin{document}

\title[Attached prime ideals over skew Ore polynomials]{Attached prime ideals over skew Ore polynomials}



\author{Sebasti\'an Higuera}
\address{Universidad Nacional de Colombia - Sede Bogot\'a}
\curraddr{Campus Universitario}
\email{sdhiguerar@unal.edu.co}
\thanks{}

\author{Armando Reyes}
\address{Universidad Nacional de Colombia - Sede Bogot\'a}
\curraddr{Campus Universitario}
\email{mareyesv@unal.edu.co}
\thanks{}


\thanks{The authors were supported by the research fund of Department of Mathematics, Faculty of Science, Universidad Nacional de Colombia - Sede Bogot\'a, Colombia, HERMES CODE 53880.}



\date{}

\dedicatory{Dedicated to Martha Rinc\'on}

\begin{abstract}

In this paper, we investigate the attached prime ideals of inverse polynomial modules over skew Ore polynomials.
\end{abstract}

\maketitle


\section{Introduction}\label{ch0}

Throughout the paper, every ring $R$ is associative (not necessarily commutative) with identity. If $R$ is commutative, then it is denoted by $K$. If $N_R$ is a right module, then the {\em right annihilator} of $N_R$ is defined by ${\rm ann}_R(N)=\left \{r \in R\mid Nr = 0\right \}$, and $N_R$ is said to be {\em prime} if $N_R \neq 0$ and ${\rm ann}_R(N) = {\rm ann}_R(N')$ for all submodule $N_R'$ of $N_R$. If $M_R$ is a right module, then a right prime ideal $P$ of $R$ is called {\em associated} of $M_R$ if there exists a prime submodule $N_R$ of $M_R$ such that $P = {\rm ann}_R(N)$. The set of all associated prime ideals of $M_R$ is denoted by ${\rm Ass}(M_R)$ \cite[p. 86]{Lam1998}. These ideals have been widely studied in the literature. For instance, Brewer and Heinzer \cite{BrewerHeinzer1974} showed that the associated prime ideals of the commutative polynomial ring $K[x]$ are all extended, that is, every $P \in {\rm Ass}(K[x]_{K[x]})$ may be expressed as $P = Q[x]$, where $Q \in {\rm Ass}(K_K)$ (see also Faith \cite{Faith2000}). Annin \cite{AnninPhD2002, Annin2002, Annin2004} extended this result to the setting of skew polynomial rings in the sense of Ore \cite{Ore1931, Ore1933}, while Nordstrom \cite{Nordstrom2005, Nordstrom2012} computed the associated prime ideals of simple torsion modules over generalized Weyl algebras defined by Bavula \cite{Bavula1992}. Later, Ouyang and Birkenmeier \cite{OuyangBirkenmeier2012} defined the {\em nilpotent associated primes} as a generalization of the associated prime ideals and described these ideals over skew polynomial rings. In the setting of the {skew PBW extensions} introduced by Gallego and Lezama \cite{GallegoLezama2011}, Ni\~no et al. \cite{NinoRamirezReyes2020} characterized the associated primes of modules over these rings, Later, Higuera et al. \cite{HigueraRamirezReyes2024, HigueraReyes2023} studied the nilpotent associated prime ideals of a skew PBW extension and investigated the associated prime ideals of induced modules over this kind of noncommutative rings.

Given the importance of primary decomposition theory and its relationships with associated prime ideals, Macdonald \cite{Macdonald1973} considered a dual theory to the primary decomposition which is commonly referred as {\em secondary representation} where the main ideals of this theory are called {\it attached primes}. According to Baig \cite{Baig2009}, $M_K$ is called a {\it secondary module} if $M_K \neq 0$ and the endomorphism $\phi_r$ of $M_K$ defined by $\phi_r(m):= mr$ for all $m \in M_K$ is either surjective or nilpotent (that is, there exists $k \in \mathbb{N}$ such that $\phi_r^k=0$), for each $r \in K$ \cite[Definition 3.1.1]{Baig2009}. Secondary modules are also called $P$-{\em secondary} since their nilradical is a prime ideal $P$ of $K$ \cite[Claim 3.1.2]{Baig2009}. If $M_K$ has a secondary representation, that is, $M_K = \sum_{i=0}^n M_i$ where each $M_i$ is secondary, then $M_K$ is called {\em representable}. If $M_i$ is $P_i$-secondary for every $1 \le i \le n$ with $P_i \neq P_j$ for $i \neq j$, and the sums $\sum_{i\neq k} M_i$ are proper submodules of $M_K$, then the representation is called {\em minimal} \cite[Definition 3.1.9]{Baig2009}. The prime ideals $P_1,\dotsc, P_n$ are called {\em attached} of $M_K$ and the set of all attached prime ideals of $M_K$ is denoted by ${\rm Att}^{*}(M_K)$ \cite[Definition 3.2.2]{Baig2009}.

Melkersson \cite{Melkersson1998} studied the attached prime ideals over commutative polynomial extensions. He investigated when multiplication by $f(x) \in K[x]$ defines a surjective endomorphism over the module $M[x^{-1}]_{K[x]}$ which consists of all the polynomials of the form $m(x)=m_0 + m_1x^{-1}+ \cdots + m_kx^{-k}$ with $m_i \in M_K$ for all $0 \le i\le k$. He showed that for $g$ and $h$ endomorphisms of $M_K$ such that $gh = hg$, if $g$ is surjective and $h$ is nilpotent then $f:= g + h$ is a surjective endomorphism \cite[Lemma 2.1]{Melkersson1998}, and with this result, he proved that if $M_K$ is Artinian or has a secondary representation, then $M = c(f)M$ where $c(f)$ is the ideal generated by the coefficients of $f(x)$. As a corollary, he obtained a characterization of the attached prime ideals of the right module $M[x^{-1}]_{K[x]}$ by showing that if $Q \in {\rm Att}^{*}(M[x^{-1}]_{K[x]})$ then we obtain that $Q = P[x]$, where $P \in {\rm Att}^{*}(M_K)$ \cite[Corollary 2.3]{Melkersson1998}. 

Annin \cite{Annin2008} introduced the concept of {\em attached prime ideal} for arbitrary modules (not necessarily representables), and provided an extension of Macdonald's theory of secondary representation to the noncommutative setting. Following Annin \cite{Annin2008}, $N_R$ is called {\it coprime} if $N_R \neq 0$ and ${\rm ann}_R(N) = {\rm ann}_R(Q)$ for all non-zero quotient module $Q_R$ of $N_R$ \cite[Definition 2.1]{Annin2008}. A right prime ideal $P$ of $R$ is called {\it attached} of $M_R$ if there exists a coprime quotient module $Q_R$ of $M_R$ such that $P = {\rm ann}_R(Q)$. The set of attached prime ideals of $M_R$ is denoted by ${\rm Att}(M_R)$ \cite[Definition 2.3]{Annin2008}. Annin \cite{Annin2011} defined 
the {\em completely $\sigma$-compatible modules} and extended Melkersson's result to the noncommutative setting. He proved that if $Q \in {\rm Att}(M[x^{-1}]_{S})$ then $Q = P[x]$ where $P \in {\rm Att}(M_R)$ and $S$ is the {skew polynomial ring} $R[x;\sigma]$ with $\sigma$ an automorphism of $R$ \cite[Theorem 3.2]{Annin2011}.

Cohn \cite{Cohn1961} introduced the {\em skew Ore polynomials of higher order} as a generalization of the skew polynomial rings considering the relation $xr := \Psi_1(r)x + \Psi_2(r)x^2 + \dotsb$ for all $r \in R$, where the $\Psi$'s are endomorphisms of $R$. Following Cohn's ideas, Smits \cite{Smits1968} introduced the ring of skew Ore polynomials of higher order over a division ring $D$ and commutation rule defined by
         \begin{align}
           xr:= r_1x + \cdots + r_kx^k,\ \text{for all}\ r\in R\ \text{and}\ k \ge 1.\label{eq:smits}  
         \end{align}
The relation (\ref{eq:smits}) induces a family of endomorphisms $\delta_1, \ldots, \delta_k$ of the group $(D,+)$ with $\delta_i(r):=r_i$ for every $1\le i\le k$ \cite[p. 211]{Smits1968}. Smits proved that if $\{\delta_2, \ldots, \delta_k \}$ is a set of left $D$-independient endomorphisms (i.e., if $c_2\delta_2(r)+ \cdots + c_k\delta_k(r)=0$ for all $r \in D$ then $c_i=0$ for all $2 \le i \le k$ \cite[p. 212]{Smits1968}), then $\delta_1$ is a injective endomorphism \cite[p. 213]{Smits1968}. There exist some algebras such as Clifford algebras, Weyl-Heisenberg algebras, and Sklyanin algebras, in which this commutation relation is not sufficient to define the noncommutative structure of the algebras since a free non-zero term $\Psi_0$ is required. Maksimov \cite{Maksimov2000} considered the skew Ore polynomials of higher order with free non-zero term $\Psi_0(r)$ where $\Psi_0$ satisfies the relation $\Psi_0(rs) = \Psi_0(r)s + \Psi_1(r)\Psi_0(s) + \Psi_2(r)\Psi_0^{2}(s) + \dotsb$, for every $r, s \in R$. Later, Golovashkin and Maksimov \cite{Golovashkinetal2005} introduced the algebras $Q(a,b,c)$ over a field $\Bbbk$ of characteristic zero with two generators $x$ and $y$, and generated by the quadratic relations $yx = ax^2 + bxy + cy^2$, where $a,b,c \in \Bbbk$. If $\{x^my^n\}$ forms a basis for $Q(a,b,c)$, then the ring generated by the quadratic relation is an algebra of skew Ore polynomials and can be defined by a system of linear mappings $\delta_0,\ldots,\delta_k$ of $\Bbbk[x]$ into itself such that for any $p(x) \in \Bbbk[x]$, $yp(x) = \delta_0(p(x)) + \delta_1(p(x))y + \cdots + \delta_k(p(x))y^k$, for some $k \in \mathbb{N}$. 

Motivated by Annin's research \cite{Annin2011} about the attached prime ideals of $M[x^{-1}]_S$ and the importance of the algebras of skew Ore polynomials of higher order, in this paper we introduce a family of noncommutative rings called {\em skew Ore polynomials} and we study the attached prime ideals of the inverse polynomial module over these rings. Since some of its ring-theoretical, homological and combinatorial properties have been investigated recently (e.g., \cite{ChaconReyes2023, Ninoetal2024, NinoReyes2023} and references therein), this article can be considered as a contribution to research on skew Ore polynomials of higher order.

The paper is organized as follows. Section \ref{Definitions} establishes some preliminaries and key results about skew Ore polynomials. In Section \ref{Attachedprime}, we introduce the completely $(\sigma,\delta)$-compatible modules and present original results (Propositions \ref{Completelycompatiblemodules}, \ref{Propertycompletely} and \ref{completelycompatible}). Under compatibility conditions, we also characterize the attached prime ideals of the right module $M[x^{-1}]_A$ where $A$ is a skew Ore polynomial ring (Theorems \ref{AnninTheorem2.1} and \ref{AnninTheorem3.2}). As expected, our results extend those above corresponding to skew polynomial rings of automorphism type presented by Annin \cite{Annin2008, Annin2011}. Finally, we present some ideas for a future work.

The symbols $\mathbb{N}$, $\mathbb{Z}$, $\mathbb{R}$, and $\mathbb{C}$ denote the set of natural numbers including zero, the ring of integer numbers, the fields of real numbers and the complex numbers, respectively. The term module will
 always mean right module unless stated otherwise. The symbol $\Bbbk$ denotes a field and $\Bbbk^{*} := \Bbbk\ \backslash\ \{0\}$.

\section{Preliminaries}\label{Definitions}

If $\sigma$ is an endomorphism of $R$, then a map $\delta : R \rightarrow R$ is called a {\em $\sigma$-derivation} of $R$ if it is additive and satisfies that $\displaystyle \delta(r s) = \sigma(r )\delta(s)+\delta(r )s$, for every $r,s \in R$ \cite[p. 26]{GoodearlWarfield2004}. Following Ore \cite{Ore1931, Ore1933}, the {\em skew polynomial ring} (also called {\em Ore extension} of $R$) over $R$ is defined as the ring $R[x;\sigma,\delta]$ generated by $R$ and $x$ such that it is a free left $R$-module with basis $\left \{x^k\ | \ k \in \mathbb{N} \right \}$ and $xr := \sigma(r)x + \delta(r)$ for every $r \in R$ \cite[p. 34]{GoodearlWarfield2004}.  

A derivation $\delta$ of $R$ is called {\em locally nilpotent} if for all $r \in R$ there exists $n(r)\ge 1$ such that $\delta^{n(r)}(r) = 0$ \cite[p. 11]{Freudenburg2006}. Following the ideas of Cohn \cite{Cohn1961} and Smits \cite{Smits1968}, we introduce the following kind of skew Ore polynomials of higher order.

\begin{definition}\label{importantdefinition}
    If $R$ is a ring, $\sigma$ is an automorphism of $R$ and $\delta$ is a locally nilpotent $\sigma$-derivation of $R$, then we define the {\em skew Ore polynomial ring} $A:=R(x; \sigma,\delta)$ which consists of the uniquely representable elements $r_0 + r_1x + \cdots + r_kx^k$ where $r_i \in R$ and $k \in \mathbb{N}$, with the commutation rule $xr:= \sigma(r)x + x\delta(r)x$ for all $r \in R$.
\end{definition}

According to Definition (\ref{importantdefinition}), if $r\in R$ and $\delta^{n(r)}(r) = 0$ for some $n(r)\ge 1$, then
\begin{align}
    xr= \sigma(r)x + \sigma\delta(r)x^2 + \cdots + \sigma\delta^{n(r)-1}(r)x^{n(r)}.\label{eq:our}
\end{align}
If we define the endomorphisms $\Psi_i:=\sigma\delta^{i-1}$ for all $i\ge 1$ and $\Psi_0:=0$, then $A$ is a skew Ore polynomial of higher order in the sense of Cohn \cite{Cohn1961}.   

\begin{example}\label{ExampleskewOre}
We present some examples of skew Ore polynomial rings.
    \begin{enumerate}
        \item If $\delta=0$ then $xr=\sigma(r)x$ and thus $R(x;\sigma)=R[x;\sigma]$ is the skew polynomial ring where $\sigma$ is an automorphism of $R$.
        \item The {\em quantum plane} $\Bbbk_q[x,y]$ is the free algebra generated by $x, y$ over $\Bbbk$, and subject to the commutation rule $xy=qyx$ with $q \in \Bbbk^*$ and $q\neq 1$. We note that $\Bbbk_q[x,y] \cong \Bbbk[y](x;\sigma)$ where $\sigma(y):=qy$ is an automorphism of $\Bbbk[y]$.
        \item The {\em Jordan plane} $\mathcal{J}(\Bbbk)$ defined by Jordan \cite{Jordan2001} is the free algebra generated by the indeterminates $x, y$ over $\Bbbk$ and the relation $yx = xy + y^2$. This algebra can be a written as the skew polynomial ring $\Bbbk[y][x;\delta]$ with $\delta(y):=-y^2$. On the other hand, notice that $\delta(x)= 1$ is a locally nilpotent derivation of $\Bbbk[x]$, and thus the Jordan plane also can be interpreted as $\Bbbk[x](y;\delta)$.
        \item D\'iaz and Pariguan \cite{DiazPariguan2009} introduced the {\em $q$-meromorphic Weyl algebra} $MW_q$ as the algebra generated by $x,y$ over $\mathbb{C}$, and defining relation $yx=qxy + x^2$, for $0 < q < 1$. Lopes \cite{Lopes2023} showed that using the generator $Y =y+(q-1)^{-1}x$ instead of $y$, it follows that $Yx = qxY$ and thus the algebra $MW_q$ can be written as a quantum plane $\mathbb{C}_q[x,y]$ \cite[Example 3.1]{Lopes2023}. Following examples (2) and (4), we conclude the algebra $MW_q$ is a skew Ore polynomial ring.
        \item Consider the algebra $Q(0,b,c)$ defined by Golovashkin and Maksimov \cite{Golovashkinetal2005} with $a=0$. It is straightforward to see that $\sigma(x)=bx$ is an automorphism of $\Bbbk[x]$ with $b \neq 0$, $\delta(x)=c$ is a locally nilpotent $\sigma$-derivation of $\Bbbk[x]$ and so $Q(0,b,c)$ can be interpreted as $A=\Bbbk[x](y;\sigma,\delta)$.
        \item If $\delta_1$ is an automorphism of $D$ and $\{\delta_2, \ldots, \delta_k \}$ is a set of left $D$-independient endomorphism, then $\delta:=\delta_1^{-1} \delta_2$ is a $\delta_1$-derivation of $D$, $\delta_{i+1}(r)= \delta_1\delta^{i}(r)$, and $\delta^{k}(r)=0$ for all $r \in D$ \cite[p. 214]{Smits1968}, and thus (\ref{eq:smits}) coincides with (\ref{eq:our}). In this way, the skew Ore polynomial rings of higher order defined by Smits can be seen as $D(x;\delta_1,\delta)$.
    \end{enumerate}
\end{example}

We need to investigate the localization technique to define the module $M[x^{-1}]_A$ (Section \ref{Attached}). In the localization of noncommutative rings the {\em Ore condition} plays an important role. A multiplicative subset $X$ of $R$ satisfies the {\em left Ore condition} if $Xr \cap Rx \neq \emptyset$ for every $r \in R$ and $x \in X$. If $X$ satisfies the left Ore condition then $X$ is called a {\em left Ore set}. Proposition \ref{Orecondition} shows that the set containing all powers of $x$ satisfies the left Ore condition.

\begin{proposition}\label{Orecondition}
    $X=\{x^k\ | \ k \ge 0 \}$ is a left Ore set of the algebra $A$.
\end{proposition}
\begin{proof}
    It is clear that $X$ is a multiplicative subset of $A$, so we have to show that $X$ satisfies the left Ore condition. Let $a= r_0 + r_1x + \cdots + r_kx^k$ be an element of $A$ with $r_k \neq 0$. Since $\delta$ is locally nilpotent, for each $r_i$ in the expression of $a$ there exists $m_i \ge 0$ such that
    \[ xr_i := \sum_{j=1}^{m_i} \sigma(\delta^{j-1}(r_i))x^{j}=a_ix,\] 
    where $a_i:= \sigma(r_i) + \sigma(\delta(r_i))x+ \cdots + \sigma(\delta^{m_i - 1}(r_i))x^{m_i -1} \in A$. In this way, for each $r_i$ there exists $a_i \in A$ such that $xr_i = a_ix$ for some $a_i \in A$, and so $xa = a'x$ for some $a' \in A$. By induction on $p$, assume that for any $a \in A$ and $x^p \in X$ there exists $\overline{a}\in A$ such that $x^{p}a = \overline{a} x$. Thus, $x^{p+1}a=x\overline{a}x$ and since $xa = \overline{a}'x$ for some $\overline{a}' \in A$, we get $x^{p+1}a= a''x$ with $a''= \overline{a}'x \in A$. Hence, $X$ is a left Ore set of $A$. 
\end{proof}
 
 By Proposition \ref{Orecondition}, we can localize $A$ by $X$, and so we denote this localization by $X^{-1}A$. It is straightforward to see that the indeterminate $x^{-1}$ satisfies the relation $x^{-1}r:= \sigma'(r)x^{-1} + \delta'(r)$, for all $r \in R$ with $\sigma'(r):=\sigma^{-1}(r)$ and $\delta'(r):=-\delta\sigma^{-1}(r)$. Dumas \cite{Dumas1991} studied the field of fractions of $D[x;\sigma,\delta]$ where $\sigma$ is an automorphism of $D$ and stated that one technique for this purpose is to consider it as a subfield of a certain field of series \cite[p. 193]{Dumas1991}. According to Dumas, if $Q$ is the field of fractions of $D[x;\sigma,\delta]$ then $Q$ is a subfield of the {\em field of series of Laurent} $D((x^{-1};\sigma^{-1}, -\delta\sigma^{-1}))$ whose elements are of the form $r_{-k}x^{-k}+\cdots + r_{-1}x^{-1} + r_0 + r_1x+ \cdots$ for some $k \in \mathbb{N}$, and satisfies the commutation rule
\begin{align*}
    xr&:= \sigma(r)x + \sigma\delta(r)x^2 + \cdots = \sigma(r)x +x\delta(r)x,\ \text{and}\\
     x^{-1}r&:= \sigma'(r)x^{-1} + \delta'(r),\ \text{for all}\ r\in R.
\end{align*}
By Definition \ref{importantdefinition} and Proposition \ref{Orecondition}, if $\sigma$ is an automorphism of $D$ and $\delta$ is a locally nilpotent $\sigma$-derivation of $D$ then $X^{-1}A \subseteq D((x^{-1};\sigma^{-1}, -\delta\sigma^{-1}))$ (see \cite{AlevDumas1995, DumasMartin2023, Dumas1992} for more details about fields of series  of Laurent).

\begin{remark}
    Following Lam et al. \cite[p. 2468]{Lametal1997}, if $\sigma$ is an automorphism of $R$ and $\delta$ is a $\sigma$-derivation of $R$, then we denote by $f_j^i$ the endomorphism of $R$ which is the sum of all possible words in $\sigma',\delta'$ built with $i$ letters $\sigma'$ and $j-i$ letters $\delta'$, for $i \leq j$. In particular, $f_0^0 = 1$, $f_j^j = \sigma'^{j}$, $f_j^0 = \delta'^j$, and $f_j^{j-1} = \sigma'^{j-1}\delta' +\sigma'^{j-2}\delta'\sigma' +\cdots+\delta'\sigma'^{j-1}$; if $\delta\sigma=\sigma\delta$, then $f_j^{i}= \binom{j}{i} \sigma'^i\delta'^{j-i}$. If $r \in R$ and $k \in\mathbb{N}$, then the following formula holds:
\begin{align}
    \displaystyle x^{-k}r =\sum_{i=0}^kf_{k}^{i}(r)x^{-i}.\label{relacion2}
\end{align}

In addition, if $r,s \in R$ and $k, k' \in \mathbb{N}$ then
\begin{align}
    \displaystyle (rx^{-k})(sx^{-k'}) = \sum_{i=0}^{k}rf_k^i(s)x^{-(k+k')}. \label{relacion3}
\end{align}
\end{remark}

Taking into account the usual addition of polynomials and the product induced by (\ref{relacion2}) and (\ref{relacion3}), we define the ring of polynomials in the indeterminate $x^{-1}$ with coefficients in $R$ and denote it by $R[x^{-1}]$. If $M_R$ is a right module then the {\em inverse polynomial module} $M[x^{-1}]_R$ is defined as the set of all polynomials of the form $f(x)=m_0 + \cdots + m_kx^{-k}$ with $m_i \in M_R$ for all $i$, the usual addition of polynomials and the action of $R$ 
 over any monomial $mx^{-k}$ is defined by (\ref{relacion2}) as follows:
\begin{align}
    \displaystyle mx^{-k}r :=\sum_{i=0}^kmf_{k}^{i}(r)x^{-i},\ \text{for all}\ m\in M_R\ \text{and}\ r\in R. 
\end{align}

\begin{remark}
We use expressions as $m(x)=m_0 + m_1x^{-1} + \cdots +m_kx^{-k}\in M[x^{-1}]_R$. With this notation, we define the \textit{leading monomial} of $m(x)$ as ${\rm lm}(m(x)):=x^{-k}$, the \textit{leading coefficient} of $m(x)$ by ${\rm lc}(m(x)):=m_k$, and the \textit{leading term} of $m(x)$ as ${\rm lt}(m(x)):=m_kx^{-k}$. The {\em negative degree} of $x^{-k}$ is defined by $\deg(x^{-k}):= -k$ for any $k \in \mathbb{N}$, and $\deg(m(x)):={\rm max}\{\deg(x^{-i})\}_{i=0}^k$ for all $m(x)\in M[x^{-1}]_R$. For any element $m(x) \in M[x^{-1}]_R$, we denote by $C_{m}$ the set of all coefficients of $m(x)$.
\end{remark}

\section{Completely compatible rings and attached prime ideals}\label{Attachedprime}

\subsection{Completely $(\sigma,\delta)$-compatible rings} Annin \cite{Annin2004} (c.f. Hashemi and Moussavi \cite{HashemiMoussavi2005}) introduced the notion of compatibility with the aim of studying the associated prime ideals of modules over skew polynomial rings. If $\sigma$ is an endomorphism of $R$ and $\delta$ is a $\sigma$-derivation of $R$, then $M_R$ is called $\sigma$-{\em compatible} if for each $m\in M_R$ and $r\in R$, $mr = 0$ if and only if $m\sigma(r)=0$; $M_R$ is $\delta$-{\em compatible} if for each $m \in M_R$ and $r\in R$, $mr = 0$ implies $m\delta(r)=0$; if $M_R$ is both $\sigma$-compatible and $\delta$-compatible, then $M_R$ is a ($\sigma,\delta$)-{\em compatible module} \cite[Definition 2.1]{Annin2004}. Continuing with the compatibility conditions on modules, $M_R$ is called {\it completely $\sigma$-compatible} if $(M/N)_R$ is $\sigma$-compatible, for every submodule $N_R$ of $M_R$ \cite[Definition 1.4]{Annin2011}. Since completely $\sigma$-compatible modules were used by Annin \cite{Annin2011} in their research about attached prime ideals of $M[x^{-1}]_S$, it is to be expected that we have to consider the {\em completely $(\sigma,\delta)$-compatible modules} to characterize this type of ideals in the setting of the skew Ore polynomials. In this way, we consider the following definition.

\begin{definition}
If $\sigma$ is an endomorphism of $R$ and $\delta$ is a $\sigma$-derivation of $R$ then $M_R$ is called {\em completely $\sigma$-compatible} if for all submodule $N_R$ of $M_R$, $(M/N)_R$ is a $\sigma$-compatible module; $M_R$ is {\em completely $\delta$-compatible} if for every submodule $N_R$ of $M_R$, $(M/N)_R$ is $\delta$-compatible; $M_R$ is said to be {\em completely $(\sigma,\delta)$-compatible} if it is both completely $\sigma$-compatible and $\delta$-compatible.
\end{definition}

\begin{example}
\begin{itemize}
    \item[{\rm (1)}] If $M_R$ is simple and $(\sigma,\delta)$-compatible, then it is not difficult to see that $M_R$ is completely $(\sigma,\delta)$-compatible. 
    \item[{\rm (2)}] Let $K$ be a local ring with maximal ideal $\mathfrak{m}$ and $\sigma$ any automorphism of $K$. Annin \cite{AnninPhD2002} proved that $M_K := K/\mathfrak{m}$ is a $\sigma$-compatible module, and since $M_K$ is simple it follows that $M_K$ is completely $\sigma$-compatible \cite[Example 3.35]{AnninPhD2002}. Additionally, if $\delta$ is a $\sigma$-derivation of $K$ such that $\delta(r) \in \mathfrak{m}$ for every $r \in \mathfrak{m}$, then $M_K$ is completely $\delta$-compatible. Indeed, if $\overline{0} \neq \overline{s} \in M_K$ and $r \in K$ satisfy that $\overline{s}r = 0$ then $sr \in \mathfrak{m}$, and since $s \notin \mathfrak{m}$ we obtain $r \in \mathfrak{m}$. If $\delta(r) \in \mathfrak{m}$ for every $r \in \mathfrak{m}$, it follows that $s\delta(r) \in \mathfrak{m}$ and so $\overline{s}\delta(r) = 0$. Therefore, $M_K$ is $\delta$-compatible and thus $M_K$ is completely $\delta$-compatible. 
\end{itemize}
\end{example}

The following proposition presents some properties of completely $(\sigma,\delta)$-compatible modules. These properties are required to prove some results of the paper.

\begin{proposition}\label{Propertycompletely}
If $M_R$ is completely $(\sigma, \delta)$-compatible and $N_R$ is a submodule of $M_R$ then the following assertions hold:
\begin{enumerate}
    \item[{\rm (1)}] If $ma \in N_R$ then $m\sigma^i(a), m\delta^{j}(a) \in N_R$ for each $i,j \in \mathbb{N}$.
    \item[{\rm (2)}] If $mab \in N_R$ then $m\sigma(\delta^{j}(a))\delta(b), m\sigma^{i}(\delta(a))\delta^{j}(b) \in N_R$ for all $i,j \in \mathbb{N}$. In particular, $ma\delta^{j}(b), m\delta^{j}(a)b \in N_R$ for all $j \in \mathbb{N}$.
    \item[{\rm (3)}] If $mab \in N_R$ or $m\sigma(a)b \in N_R$ then $m\delta(a)b \in N_R$.
\end{enumerate}
\end{proposition}
\begin{proof} If $M_R$ is completely $(\sigma,\delta)$-compatible, then $(M/N)_R$ is $(\sigma,\delta)$-compatible. Considering the elements $\overline{ma}=\overline{mab}=\overline{0} \in (M/N)_R$, the assertions follow from \cite[Lemma 2.15]{AlhevazMoussavi2012}.
\end{proof}

Annin proved some properties of completely $\sigma$-compatible modules \cite[p. 539]{Annin2011}. Proposition \ref{Completelycompatiblemodules} extends these statements for completely $(\sigma,\delta)$-compatible modules.

\begin{proposition}\label{Completelycompatiblemodules} Let $\sigma$ be an endomorphism of $R$ and $\delta$ a $\sigma$-derivation of $R$.
\begin{itemize}
    \item[\rm (1)] If $M_R$ is completely $(\sigma,\delta)$-compatible then $M_R$ is $(\sigma,\delta)$-compatible. 
    \item[\rm (2)] If $M_R$ is a completely $(\sigma, \delta)$-compatible module then $(M/N)_R$ is completely $(\sigma, \delta)$-compatible, for every submodule $N_R$ of $M_R$.
\end{itemize}
\end{proposition}
\begin{proof}
\begin{itemize}
    \item[(1)] If $M_R$ is a completely $(\sigma,\delta)$-compatible module then $(M/N)_R$ is a $(\sigma,\delta)$-compatible module, for every submodule $N_R$ of $M_R$. In particular, $(M/\{0\})_R \cong M_R$ is $(\sigma,\delta)$-compatible for the submodule $\{0\}_R$ of $M_R$.
    \item[(2)] Suppose that $M_R$ is completely $(\sigma,\delta)$-compatible and consider a submodule $N'/N$ of $(M/N)_R$ where $N'_R$ is a submodule of $M_R$ with $N \subsetneq N'$. By the third isomorphism theorem for modules, $((M/N)/(N'/N))_R \cong (M/N')_R$, and since $M_R$ is completely $(\sigma,\delta)$-compatible, we have that $(M/N')_R$ is $(\sigma,\delta)$-compatible and thus $((M/N)/(N'/N))_R$ is $(\sigma,\delta)$-compatible, whence $(M/N)_R$ is a completely $(\sigma,\delta)$-compatible module.
\end{itemize}
\end{proof}

We present other important property of completely $(\sigma,\delta)$-compatible modules.

\begin{proposition}\label{completelycompatible} If $\sigma$ is bijective and $M_R$ is a completely $(\sigma,\delta)$-compatible module, then $M_R$ is a completely $(\sigma',\delta')$-compatible module.
\end{proposition}
\begin{proof}
Assume that $M_R$ is completely $(\sigma,\delta)$-compatible. If $N_R$ is a submodule of $M_R$ then $mr \in N_R$ if and only if $m\sigma(r)\in N_R$, for all $m\in M_R$ and $r\in R$. Since $\sigma$ is bijective we have $m\sigma^{-1}(r) \in N_R$ if and only if $mr\in N_R$, and so $(M/N)_R$ is a $\sigma'$-compatible proving that $M_R$ is completely $\sigma'$-compatible. On the other hand, if $M_R$ is completely $\sigma'$-compatible then $mr \in N_R$ which implies that $m\sigma^{-1}(r)\in N_R$. If $M_R$ is completely $\delta$-compatible and $m\sigma^{-1}(r)\in N_R$, then $m\delta\sigma^{-1}(r)$ and thus $M_R$ is completely $\delta'$-compatible. Therefore $M_R$ is completely $(\sigma',\delta')$-compatible.
\end{proof}

\subsection{Attached prime ideals}\label{Attached} 

In this section, we define an $A$-module structure to the inverse polynomial module $M[x^{-1}]_R$ and study the attached prime ideals of the module $M[x^{-1}]_A$. The action of $A$ over $M[x^{-1}]_R$ is given by
    \begin{align}
        mx^{-1}r &:= m\sigma'(r)x^{-1} + m\delta'(r)\ \text{for all}\ r \in R\ \text{and}\ m \in M_R,\ \text{and} \label{eqn:(3.1)}\\
         x^{-i}x^{j}&:=x^{-i+j}\ \text{if}\ j \le i\ \text{and}\ 0\ \text{otherwise} \label{eqn:(3.2)}.
    \end{align}

\begin{remark} By (\ref{eqn:(3.1)}) and (\ref{eqn:(3.2)}), if $\delta:= 0$ then $mx^{-i}rx^j:=m\sigma'^{i}(r)x^{-i+j}$ for all $r \in R$ and $i,j \in \mathbb{N}$ with $j \le i$, which coincides with $M[x^{-1}]_S$ \cite[p. 538]{Annin2011}. 
\end{remark}

If $N_R$ is a right module and $N[x^{-1}]_R$ (resp., $N[x^{-1}]_A$) is a right module, then the right annihilator is denoted by ${\rm ann}_R(N[x^{-1}])$ (resp., ${\rm ann}_A(N[x^{-1}])$). The following lemma characterizes the ideals generated by right prime ideals of $A$ that correspond to annihilators of quotient modules of $M[x^{-1}]_A$.

\begin{lemma}\label{miniTheorem2.1}
If $M_R$ is completely $(\sigma, \delta)$-compatible and $P$ is a right prime ideal of $R$ such that $P = {\rm ann}_R(M/N)$ for some submodule $N_R$ of $M_R$, then 
\[
PA = {\rm ann}_A(M[x^{-1}]/N[x^{-1}]).
\]
\end{lemma}
\begin{proof}
    By Propositions \ref{Propertycompletely} and \ref{completelycompatible}, if $M_R$ is completely $(\sigma, \delta)$-compatible then $C_{mf} \subseteq N$ for every $m(x) \in M[x^{-1}]_A$, $f(x)\in PA$, and thus $m(x)f(x) \in N[x^{-1}]_A$ proving that $PA \subseteq {\rm ann}_A(M[x^{-1}]/N[x^{-1}])$. For the other inclusion, if $f(x) \notin PA$ then there exists a monomial $r_lx^l$ of $f(x)$ such that $r_l \notin P$ for some $0 \le l \le j$. So there exists $m \in M_R$ such that $mr_l \notin N_R$, which implies that $mf(x) \notin N[x^{-1}]_A$ and so $f(x) \notin {\rm ann}_A(M[x^{-1}]/N[x^{-1}])$ whence ${\rm ann}_A(M[x^{-1}]/N[x^{-1}]) \subseteq PA$.
\end{proof}

Theorem \ref{AnninTheorem2.1} shows that right ideals of $A$ generated by attached prime ideals of $M_R$ are attached prime ideals of $M[x^{-1}]_A$, and extends \cite[Theorem 2.1]{Annin2011}.

\begin{theorem}\label{AnninTheorem2.1}
If $M_R$ is completely $(\sigma, \delta)$-compatible then
\[
\displaystyle {\rm Att} (M[x^{-1}]_A) \supseteq \left \{ PA \ | \ P\in {\rm Att}(M_R) \right \}.    
\]
\end{theorem}
\begin{proof}
If $P$ is an attached prime ideal of $M_R$ and $(M/N)_R$ is the quotient coprime of $M_R$ such that $P = {\rm ann}_R(M/N)$ for some submodule $N_R$ of $M_R$, it follows that $PA={\rm ann}_A(M[x^{-1}]/N[x^{-1}])$ by Lemma \ref{miniTheorem2.1}. Let us prove that $(M[x^{-1}]/N[x^{-1}])_A$ is a quotient coprime. If $M[x^{-1}]/N[x^{-1}]\neq 0$ then there exists a submodule $Q_A$ of $M[x^{-1}]_A$ such that $Q \supsetneq N[x^{-1}]$. Let $C_Q$ be the subset of $M$ that consists of all coefficients of the elements of $Q$ and consider $Q_R'$ the submodule of $M_R$ generated by $C_Q$. Notice that if $Q \neq M[x^{-1}]$ then $Q' \neq M$, and if $(M/N)_R$ is a coprime module then $P={\rm ann}_R(M/N)={\rm ann}_R(M/Q')$. 

If $g(x) = r_0 + \cdots + r_jx^j \in {\rm ann}_A(M[x^{-1}]/Q)$ then $f(x)g(x) \in Q_A$ and hence $C_{fg} \subseteq Q_R'$ for all $f(x) =  m_0 + \cdots + m_kx^{-k} \in M[x^{-1}]_A$. By Propositions \ref{Propertycompletely} and \ref{completelycompatible}, if $M_R$ is completely $(\sigma, \delta)$-compatible then $m_ir_j \in Q_R'$, whence $m_ir_j \in N_R$ for all $i,j$. Since $m_ir_j \in N_R$, we have that $m(x)f(x) \in N[x^{-1}]_A$ by Propositions \ref{Propertycompletely} and \ref{completelycompatible}, and thus $M[x^{-1}]f(x) \subseteq N[x^{-1}]$. Therefore $f(x)\in {\rm ann}_A(M[x^{-1}]/N[x^{-1}])$ proving that ${\rm ann}_A(M[x^{-1}]/Q) \subseteq {\rm ann}_A(M[x^{-1}]/N[x^{-1}])$. 

If $f \in {\rm ann}_A(M[x^{-1}] /N[x^{-1}])$ then $M[x^{-1}] f \subseteq N[x^{-1}]$, and since $N[x^{-1}] \subsetneq Q$ we have that $M[x^{-1}]f \subsetneq Q$ showing that ${\rm ann}_A(M[x^{-1}] /N[x^{-1}])\subseteq{\rm ann}_A(M[x^{-1}]/Q)$, and thus $(M[x^{-1}]/N[x^{-1}])_A$ is a coprime module.
\end{proof}

\begin{corollary}[{\cite[Theorem 2.1]{Annin2011}}] If $M_R$ is completely $\sigma$-compatible then
\[
\displaystyle {\rm Att}(M[x^{-1}]_S) \supseteq \left \{ P[x] \ | \ P\in {\rm Att}(M_R) \right \}.    
\]
\end{corollary}

Lemma \ref{AnninLemma2.4} shows when a coprime module over the ring of skew Ore polynomials $A$ is a coprime module over $R$ and generalizes \cite[Lemma 2.4]{Annin2011}.

\begin{lemma}\label{AnninLemma2.4}
If $P_A$ is a coprime module and $P_R$ is a completely $(\sigma, \delta)$-compatible module then $P_R$ is coprime.
\end{lemma}
\begin{proof}
Consider a submodule $Q_R$ of $P_R$ and let us show that ${\rm ann}_R(P)={\rm ann}_R(P/Q)$. If $r \in {\rm ann}_R(P)$ then $Pr=0 \in Q$, and so $r \in {\rm ann}_R(P/Q)$. For the other inclusion, assume that $r \in {\rm ann}_R(P/Q)$ and $N_A:= \sum Q_A'$ where $Q_A'$ is any submodule over $A$ with $Q' \subseteq Q$. So $N_A \subseteq Q_R \subsetneq P_A$, and since $P_A$ is a coprime module then ${\rm ann}_A(P) = {\rm ann}_A(P/N)$. Let $p \in P_A$ and denotes by $prA_A$ the module generated by the element $pr$. Let us prove that $prA \subseteq N$. Since $x^{-1}r:= \sigma'(r)x^{-1} + \delta'(r)$, we have $rx = x\sigma'(r) + x \delta'(r)x$ for every $r \in R$, and thus if $f(x)=r_0 + \cdots + r_lx^l \in A$ and $Prr_j \subseteq Q$ for every $0 \le j \le l$, then $prf(x) \in Q_A$ by Propositions \ref{Propertycompletely} and \ref{completelycompatible}. Hence $prA \subseteq Q$ and so $prA \subseteq N$ by definition of $N_A$. In this way, $Pr \subseteq N$ which implies that $r \in {\rm ann}_A(P/N)={\rm ann}_A(P)$ proving that ${\rm ann}_R(P)={\rm ann}_R(P/Q)$ for every proper submodule $Q_R$ of $P_R$, that is, $P_R$ is a coprime module.
\end{proof}

For any submodule $P_R$ of $M[x^{-1}]_R$, we set $P_{k} := \{m \in M\ |\  mx^{-k} \in P \}$ for each $k \in \mathbb{N}$ and denote by $\langle P_k \rangle$ the submodule of $M_R$ generated by $P_k$. Lemma \ref{AnninLemma2.5} guarantees the existence of certain maximal submodules of $M_R$ and generalizes \cite[Lemma 2.5]{Annin2011}. 

\begin{lemma}\label{AnninLemma2.5}
If $P_R$ is a maximal submodule of $M[x^{-1}]_R$, we either have $\langle P_k \rangle = M$ or else $\langle P_k \rangle$ is a maximal submodule of $M_{R}$ for each $k \in \mathbb{N}$. Additionally, there exists $k \in \mathbb{N}$ for which the latter holds.
\end{lemma}
\begin{proof}
Assume that there exists a submodule $M_R'$ such that $\langle P_k \rangle \subsetneq M' \subseteq M$ and let us see that $M'=M$. If $m' \in M_R'$ and $m \notin \langle P_k\rangle$ then $m'x^{-k} \notin P_R$, and since $P_R$ is a maximal submodule of $M[x^{-1}]_R$ we obtain $M[x^{-1}]_R = P_R + m'x^{-k}R_R$, that is, for every $f(x) \in M[x^{-1}]_R$ there exist $p \in P_R$ and $r \in R$ such that $f(x) = p + m'x^{-k}r$. If $m'x^{-k}r= m\sigma'^{k}(r)x^{-k} + mp_{k,r}$ where $p_{k,r}:=\sum_{i=0}^{k-1}f_{k}^{i}(r)x^{-i}$, consider the element $f'(x)=mx^{-k} + mp_{k,r}$. So $f'(x)=p + m'x^{-k}r$ for some $p \in P_R$ which implies that $p= (m - m'\sigma'^{k}(r))x^{-k} \in P_R$, whence $m - m'\sigma'^{k}(r) \in \langle P_{k}\rangle \subseteq M'$ and thus $m \in M_R'$. Therefore $M\subseteq M'$ and so $M'=M$ proving that $\langle P_k \rangle$ is a maximal submodule of $M_R$. If $\langle P_{k}\rangle = M$ for all $k \in \mathbb{N}$, then $P = M[x^{-1}]$ which is a contradiction. In this way, there is $k \in \mathbb{N}$ such that $\langle P_{k}\rangle$ is a maximal submodule of $M_R$. 
\end{proof}

Annin \cite{Annin2011} considered the {\em Bass modules} in his study of the attached prime ideals of $M[x^{-1}]_S$ \cite[p. 544]{Annin2011}. We recall that $M_R$ is called a {\em Bass module} if every proper submodule $N_R$ is contained in a maximal submodule of $M_R$ \cite[p. 205]{Faith1995}. Theorem \ref{AnninTheorem3.2} characterizes the attached primes of $M[x^{-1}]_A$ and extends \cite[Theorem 3.2]{Annin2011}.

\begin{theorem}\label{AnninTheorem3.2}
If $M[x^{-1}]_R$ is a completely $(\sigma, \delta)$-compatible Bass module then
\[
\displaystyle {\rm Att} (M[x^{-1}]_A) = \left \{ PA \ | \ P\in {\rm Att}(M_R) \right \}.    
\]
\end{theorem}
\begin{proof}
In view of Theorem \ref{AnninTheorem2.1}, we only need to prove
\[
\displaystyle {\rm Att}(M[x^{-1}]_A) \subseteq \left \{ PA \ | \ P\in {\rm Att}(M_R) \right \}.    
\]
Let $I \in {\rm Att}(M[x^{-1}]_A)$ and $Q_A$ be a submodule of $M[x^{-1}]_A$ such that $(M[x^{-1}]/Q)_A$ is a coprime module with $I = {\rm ann}_A(M[x^{-1}]/Q)$. It is clear that $I \cap R$ is equal to ${\rm ann}_{R}(M[x^{-1}]/Q)$. By Lemma \ref{AnninLemma2.4}, $(M[x^{-1}]/Q)_R$ is a coprime module and since $M[x^{-1}]_R$ is a Bass module, $(M[x^{-1}]/Q)_R$ contains a maximal submodule such that $P/Q \subseteq M[x^{-1}]/Q$. By the coprimality of $(M[x^{-1}]/Q)_R$, $(M[x^{-1}]/P)_R$ is coprime and $I\cap R = {\rm ann}_{R}(M[x^{-1}]/P)$. Let us prove that $I \cap R \in {\rm Att}(M_R)$ and $I=(I\cap R)A$.

If $P_R$ is a maximal submodule of $M[x^{-1}]_R$, then there exists $k \in \mathbb{N}$ such that $\langle P_k\rangle$ is a maximal submodule of $M_R$ by Lemma \ref{AnninLemma2.5}, and we can set the smallest $k$ that satisfies this hypothesis. If $\langle P_k\rangle$ is maximal then there exists $m_k \in M$ such that $m_k \notin \langle P_k \rangle$, and so $mx^{-k} \notin P_R$ whence $mx^{-k} + P$ is a cyclic generator of the simple module $(M[x^{-1}]/P)_R$. Let $\varphi$ be the map of $(M/\langle P_k\rangle)_R$ over $(M[x^{-1}]/P)_R$ given by $\varphi(m_k + \langle P_k\rangle):= m_kx^{-k} + P$. By the complete $(\sigma, \delta)$-compatibility of $M_R$ and Propositions \ref{Propertycompletely} and \ref{completelycompatible}, if $m_kr \in \langle P_k\rangle$ for some $r\in R$ then $m_k\sigma'^{k}(r) \in \langle P_k\rangle $ whence $m_k\sigma'^{k}(r)x^{-k}\in P_R$. By minimality of $k$, $m_kf_{k}^i(r)x^{-i} \in P_R$ for all $0 \le i \le k-1$, which implies that $m_kx^{-k}r \in P_R$ and hence $\varphi$ is well defined. 

Let us see that $\varphi$ is surjective. If $m_kx^{-k} + P$ generates the module $(M[x^{-1}]/P)_R$ and $\varphi(m_k + \langle P_k\rangle) := m_kx^{-k} + P$ then $\varphi$ is surjective. Let $\psi$ be the homomorphism of $M_R$ over $(M/\langle P_k\rangle)_R$ defined by $\psi(m):= m + \langle P_k\rangle$ for every $m \in M_R$. If $\varphi\circ \psi$ is a surjective homomorphism of $M_R$ over $(M[x^{-1}]/P)_R$ and
 $I \cap {R} \in {\rm Att}((M[x^{-1}]/P)_R)$, then $I \cap R \in {\rm Att}(M_R)$.

We need to prove that $I=(I\cap R)A$. Since $I$ is an ideal of $A$ we have $(I \cap R)A\subseteq I$. For the other inclusion, take an element $f(x)= r_0 + \cdots + r_jx^j \in I$ and let us see by induction that $r_i \in I \cap R$ for all $1 \le i \le j$. Notice that $(m_kx^{-k}+P)f(x)=m_kx^{-k}r_0 +  \text{lower terms} \in Q \subseteq P$, and since every monomial of the ``lower terms'' belongs to the submodule $P_R$ by minimality of $k$, we have $m_kx^{-k}r_0 \in P_R$ which implies that $r_0 \in I \cap R$. Assume that $r_0, \ldots, r_i \in I \cap R$ for some $ i \le j$, and let us prove that $r_{i+1}\in I \cap R$. If $r_0, \ldots, r_i \in I \cap R$ then $r_0 + \cdots + r_ix^i \in I$, whence $r_{i+1}x^{i+1} + \cdots + r_jx^j \in I$.
\begin{center}
    $(m_kx^{-k-i-1})(r_{i+1}x^{i+1} + \cdots + r_jx^j)= m_k\sigma'^{k+i+1}(r_{i+1})x^{-k} +  \text{lower terms}  \in Q \subseteq P$.
\end{center}
By minimality of $k$, every monomial  of the ``lower terms'' belongs to the submodule $P_R$ and thus $m_k\sigma'^{k+i+1}(r_{i+1})x^{-k} \in P_R$, and by relation $x^{-1}r:= \sigma'(r)x^{-1} + \delta'(r)$ it follows that 
\begin{center}
    $m_k\sigma'^{k+i+1}(r_{i+1})x^{-k}=m_kx^{-k}\sigma'^{i+1}(r_{i+1}) +  \text{lower terms} \in Q \subseteq P$,
\end{center}
where every monomial  of the ``lower terms'' belongs to $P_R$ by minimality of $k$. So, if $M[x^{-1}]_R$ is a completely $(\sigma, \delta)$-compatible module and $m_kx^{-k}\sigma'^{i+1}(r_{i+1}) \in P_R$, then $m_kx^{-k}r_{i+1} \in P_R$ and thus $r_{i+1} \in I \cap R$ whence $f(x)= r_0 + \cdots + r_jx^j$ belongs to $(I \cap R)A$. Therefore $I=(I \cap R)A$.

\end{proof}

\begin{corollary}[{\cite[Theorem 3.2]{Annin2011}}] \label{CorollaryAnninTheorem3.2} If $M[x^{-1}]_R$ is a completely
$\sigma$-compatible Bass module then
\[
\displaystyle {\rm Att} (M[x^{-1}]_S) = \left \{ P[x] \ | \ P\in {\rm Att}(M_R) \right \}.    
\]
\end{corollary}






\section{Examples}
The importance of our results is appreciated when we extend their application to algebraic structures that are more general than those considered by Annin \cite{Annin2011}, that is, noncommutative rings which cannot be expressed as skew polynomial rings of endomorphism type. In this section, we consider families of algebras that have been studied in the literature to exemplify the results obtained in this paper.

\begin{example}
  Let $A$ be the Jordan plane $\mathcal{J}(\Bbbk)$ or the $q$-skew Jordan plane $\mathcal{J}_q(\Bbbk)$ with $q \in \Bbbk^*$ and $q\neq 1$. Consider the right module $M[y^{-1}]_{A}$ under the action defined by (\ref{eqn:(3.1)}) and (\ref{eqn:(3.2)}). If $M_{\Bbbk[x]}$ is a right module such that $M[y^{-1}]_{\Bbbk[x]}$ is a completely $(\sigma,\delta)$-compatible module, then the characterization of the attached prime ideals of $M[y^{-1}]_{A}$ is obtained from Theorems \ref{AnninTheorem2.1} and \ref{AnninTheorem3.2}.
\end{example}

\begin{example}
  Let $A$ be the {$q$-meromorphic Weyl algebra} $MW_q$ with $yx=qxy + x^2$, for $0 < q < 1$. If $M[x^{-1}]_{\mathbb{C}[y]}$ is completely $(\sigma,\delta)$-compatible where $\sigma(y):=q^{-1}y$ and $\delta(y):=-q^{-1}$, then the characterization of the attached primes of $M[x^{-1}]_{A}$ follows from Theorems \ref{AnninTheorem2.1} and \ref{AnninTheorem3.2}. Thinking about the change of variable presented by Lopes (Example \ref{ExampleskewOre} (4)), $MW_q$ can be interpreted as the quantum plane $\mathbb{C}_q[x,y]$ with $yx=qxy$. In this way, if $M[x^{-1}]_{\mathbb{C}[y]}$ is a completely $\sigma$-compatible module with $\sigma(y):=q^{-1}y$, then the description of the attached prime ideals over $M[x^{-1}]_{A}$ follows from Theorems \ref{AnninTheorem2.1} and \ref{AnninTheorem3.2} or Corollary \ref{CorollaryAnninTheorem3.2}.
\end{example}

\begin{example}
   If $A$ is the algebra of skew Ore polynomials of higher order $Q(0,b,c)$ subject to the relation $yx=bxy + cy^2$ where $b,c\in \Bbbk^{*}$ and $M_{\Bbbk[x]}$ is a right module which satisfies that $M[y^{-1}]_{\Bbbk[x]}$ is completely $(\sigma,\delta)$-compatible, then Theorems \ref{AnninTheorem2.1} and \ref{AnninTheorem3.2} described the attached prime ideals of $M[y^{-1}]_{A}$. In a similar way, we get the characterization of these ideals over the right module $M[x^{-1}]_{A}$ when $A$ is the algebra $Q(a,b,0)$.
\end{example}


\begin{example} With the aim of constructing new Artin-Schelter regular algebras, Zhang and Zhang \cite{ZhangZhang2008} defined the {\em double Ore extensions} (or {\em double extensions}, for short) over a $\Bbbk$-algebra $R$ and presented $26$ new families of Artin-Schelter regular algebras of global dimension four. It is possible to find some similarities between the definition of double extensions and two-step iterated skew polynomial rings. Nevertheless, there exist no inclusions between the classes of all double extensions and of all length two iterated skew polynomial rings (c.f. \cite{Carvalhoetal2011}). Several researchers have investigated different relations of double extensions with Poisson, Hopf, Koszul and Calabi-Yau algebra (see \cite{RamirezReyes2024} and reference therein). We start by recalling the definition of a double extension in the sense of Zhang and Zhang, and since some typos occurred in their papers \cite[p. 2674]{ZhangZhang2008} and \cite[p. 379]{ZhangZhang2009} concerning the relations that the data of a double extension must satisfy, we follow the corrections presented by Carvalho et al. \cite{Carvalhoetal2011}.

   \begin{definition} [{\cite[Definition 1.3]{ZhangZhang2008}; \cite[Definition 1.1]{Carvalhoetal2011}}]\label{DefinitionDoubleExtension}

       If $B$ is a $\Bbbk$-algebra and $R$ is a subalgebra of $B$, then 
        \begin{itemize}
            \item[(a)] $B$ is called a {\em right double extension} of $R$ if the following conditions hold: 
        \begin{itemize}
        \item[\rm (i)] $B$ is generated by $R$ and two new variables $y_1$ and $y_2$.
        \item[\rm (ii)] $y_1$ and $y_2$ satisfy the relation
        \begin{equation}\label{eqn:aii}
        y_2y_1 = p_{12}y_1y_2 + p_{11}y_1^2 + \tau_1y_1 + \tau_2y_2 + \tau_0, 
        \end{equation}
        where $p_{12}, p_{11} \in \Bbbk$ and $\tau_1, \tau_2, \tau_0 \in R$.
        \item[\rm (iii)] $B$ is a free left $R$-module with a basis $\left \{y_1^{i}y_2^{j} \ | \ i,j \geq 0 \right \}$.
        \item[\rm (iv)] $y_1R + y_2R + R\subseteq Ry_1 + Ry_2 + R$.
      \end{itemize}
       \item[(b)] A right double extension $B$ of $R$ is called a {\em double extension} if
        \begin{itemize}
        \item[\rm (i)] $p_{12}\neq 0$.
        \item[\rm (iii)] $B$ is a free right $R$-module with a basis $\left \{y_2^{i}y_i^{j} \ | \ i,j \geq 0 \right \}$.
        \item[\rm (iv)] $y_1R + y_2R + R = Ry_1 + Ry_2 + R$.
      \end{itemize}
   \end{itemize}
   \end{definition}
    Condition (a)(iv) from Definition \ref{DefinitionDoubleExtension} is equivalent to the existence of two maps
\begin{center}
    $\sigma(r):=\begin{pmatrix} \sigma_{11}(r) & \sigma_{12}(r) \\ \sigma_{21}(r) & \sigma_{22}(r) \end{pmatrix}
    \ \text{and}\ \delta(r):= \begin{pmatrix} \delta_1(r)\\ \delta_2(r) \end{pmatrix}$ for all $r \in R$,
\end{center}
such that 
\begin{equation}\label{eqn:R2}
    \begin{pmatrix} y_1\\ y_2 \end{pmatrix}r:= \begin{pmatrix} y_1r\\ y_2r \end{pmatrix}=\begin{pmatrix} \sigma_{11}(r) & \sigma_{12}(r) \\ \sigma_{21}(r) & \sigma_{22}(r) \end{pmatrix} \begin{pmatrix} y_1\\ y_2 \end{pmatrix} + \begin{pmatrix} \delta_1(r)\\ \delta_2(r) \end{pmatrix}.
\end{equation}

 If $B$ is a right double extension of $R$ then we write $B := R_P [y_1, y_2; \sigma, \delta, \tau]$ where $P:=\{ p_{12}, p_{11} \}\subseteq \Bbbk$, $\tau:=\{\tau_1, \tau_2, \tau_0 \}\subseteq R$ and $\sigma,\delta$ are as above. The set $P$ is called a {\em parameter} and $\tau$ a {\em tail}. If $\delta:=0$ and $\tau$ consists of zero elements then the double extension is denoted by $R_P [y_1, y_2; \sigma]$ and is called a {\em trimmed double extension} \cite[Convention 1.6 (c)]{ZhangZhang2008}. It is straightforward to see that the relation (\ref{eqn:aii}) is given by 
\begin{equation}\label{eqn:RRR}
    y_2y_1 = p_{12}y_1y_2 + p_{11}y_1^2.
\end{equation}
Since $p_{12}, p_{11}\in \Bbbk$ the expression (\ref{eqn:RRR}) can be written as $y_1y_2= p_{12}^{-1}y_2y_1 - p_{12}^{-1}p_{11}y_1^2$. It is clear that $\sigma(y_2)=p_{12}^{-1}y_2$ is an automorphism of $\Bbbk[y_2]$ and $\delta(y_2)=- p_{12}^{-1}p_{11}$ is a locally nilpotent $\sigma$-derivation of $\Bbbk[y_2]$. In this way, the trimmed double extension $R_P [y_1, y_2; \sigma]$ can be seen as $A=\Bbbk[y_2](y_1;\sigma,\delta)$. If $M[y_2^{-1}]_{A}$ is a right module under the action given by  (\ref{eqn:(3.1)}) and (\ref{eqn:(3.2)}) and $M_{\Bbbk[y_2]}$ is a module such that $M[y_1^{-1}]_{\Bbbk[y_2]}$ is completely $(\sigma,\delta)$-compatible, then Theorems \ref{AnninTheorem2.1} and \ref{AnninTheorem3.2} describe the attached prime ideals of $M[y_2^{-1}]_{A}$.

\end{example}

\section{Future work}
As we mentioned in the introduction, Nordstrom \cite{Nordstrom2005, Nordstrom2012} studied the associated primes of simple torsion modules over generalized Weyl algebras which are $\mathbb{N}$-graded rings and contain the skew polynomial ring $R[x;\sigma]$ as a subring. As a possible future work, we have in mind to investigate the attached prime ideals of inverse polynomial modules over generalized Weyl algebras. 

Lezama and Latorre \cite{Lezamalatorre2017} introduced the {semi-graded rings} with the aim of generalizing $\mathbb{N}$-graded rings (such as the skew polynomial rings of endomorphism type and the generalized Weyl algebras), skew PBW extensions and many other algebras of interest in noncommutative algebraic geometry and noncommutative differential geometry. We think as future work to investigate the theory of the attached prime ideals and secondary representations over semi-graded rings.


\begin{thebibliography}{}

\bibitem{AlevDumas1995} Alev, J., Dumas, F. (1995). Automorphismes de certains complet\'es du corps de Weyl quantique. {\em Collect. Math.} 46:1--9.

\bibitem{AlhevazMoussavi2012} Alhevaz, A., Moussavi, A. (2012). On skew {A}rmendariz and skew quasi-Armendariz modules. {\em Bull. Iranian Math. Soc.} {38}(1):55--84.

\bibitem{AnninPhD2002} Annin, S. (2002). Associated and Attached Primes Over Noncommutative Rings. PhD Thesis. University of California, Berkeley.

\bibitem{Annin2002} Annin, S. (2002). Associated primes over skew polynomial ring. {\em Commun. Algebra} {30}(5):2511--2528.

\bibitem{Annin2004} Annin, S. (2004). Associated primes over Ore extension rings. {\em J. Algebra Appl.} {3}(2):193--205. 


\bibitem{Annin2008} Annin, S. (2008). Attached primes over noncommutative rings. {\em J. Pure Appl. Algebra} {212}(3):510--521.

\bibitem{Annin2011} Annin, S. (2011). Attached primes under skew polynomial extensions. {\em J. Algebra Appl.} {10}(3):537--547.



\bibitem{Baig2009} Baig, M. (2009). Primary decomposition and secondary representation of modules over a commutative ring. Master Thesis. Georgia State University.

\bibitem{Bavula1992} Bavula, V. V. (1992). Generalized Weyl algebras and their representations. {\em St. Petersburg Math. J.} {4}(1):71--92.



\bibitem{BrewerHeinzer1974} Brewer, J., Heinzer, W. (1974). Associated primes of principal ideals. {\em Duke Math. J.} {41}(1):1--7.





\bibitem{Carvalhoetal2011} Carvalho, P. A. A. B., Lopes, S. A., Matczuk, J. (2011). Double Ore Extensions Versus Iterated Ore Extensions. {\em Commun. Algebra} 39(8):2838--2848.


\bibitem{ChaconReyes2023} Chac\'on, A., Reyes, A. (2024). On the schematicness of some Ore polynomials of higher order generated by homogenous quadratic relations. {\em J. Algebra Appl.} 2550207.


\bibitem{Cohn1961} Cohn, P. M. (1961). Quadratic extensions of skew fields. {\em Proc. London Math. Soc. {\rm(}3{\rm)} } {3}(1):531--556. 


\bibitem{DiazPariguan2009} D\'iaz, R., Pariguan, E. (2009). On the $q$-meromorphic Weyl algebra. {\em S\~ao Paulo J. Math. Sci.} 3(2): 283--29

\bibitem{DumasMartin2023} Dumas, F., Martin, F. (2023). Invariants of formal pseudodifferential operator algebras and algebraic modular forms. {\em Rev. Un. Mat. Argentina} 65(1):1--31.

\bibitem{Dumas1992} Dumas, F. (1992). Skew power series rings with general commutation formula. {\em Theoret. Comput. Sci.} 98(1):99--114.

\bibitem{Dumas1991} Dumas, F. (1991). Sous-corps de fractions rationnelles des corps gauches de series de Laurent. Topics in Invariant Theory. Vol. 1478. Berlin, Heidelberg: Springer.


\bibitem{Faith2000} Faith, C. (2000). Associated primes in commutative polynomial rings. {\em Commun. Algebra} {28}(8):3983--3986.

\bibitem{Faith1995} Faith, C. (1995). Rings whose modules have maximal submodules. {\em Publ. Mat.} 39(1):201--214.


\bibitem{Freudenburg2006} Freudenburg, G. (2006). Algebraic theory of locally nilpotent derivations. Vol. 136. Berlin: Springer.

\bibitem{GallegoLezama2011} Gallego, C., Lezama, O. (2011). Gr\"obner Bases for Ideals of $\sigma$-PBW Extensions. {\em Commun. Algebra} 39(1):50--75.



\bibitem{Golovashkinetal2005} Golovashkin, A. V., Maksimov, V. M. (2005). On algebras of skew polynomials generated by quadratic homogeneous relations. {\em J. Math. Sci. {\rm(}N.Y.{\rm )}} 129(2):3757--3771.



\bibitem{GoodearlWarfield2004}Goodearl, K. R., Warfield, R. B. (2004). An Introduction to Noncommutative Noetherian Rings. Vol. 61. Cambridge University Press.



\bibitem{HashemiMoussavi2005} Hashemi, E., Moussavi, A. (2005). Polinomial extensions of quasi-Baer rings. {\em Acta Math. Hungar.} {107}(3):207--224.


\bibitem{HigueraRamirezReyes2024} Higuera, S., Ram\'irez, M. C., Reyes, A. (2024). On the uniform dimension and the associated primes of skew PBW extensions. \url{https://arxiv.org/abs/2404.18698}.

\bibitem{HigueraReyes2023} Higuera, S., Reyes, A. (2023). On weak annihilators and nilpotent associated primes of skew PBW extensions. {\em Commun. Algebra} {51}(11):4839--4861.

\bibitem{Jordan2001} Jordan, D. (2001). The Graded Algebra Generated by Two Eulerian Derivatives. {\em Algebr. Represent. Theory} 4(3):249–275.


\bibitem{Lametal1997} Lam, T. Y., Leroy, A., Matczuk, J. (1997) Primeness, semiprimeness and prime radical of Ore extensions, {\em Commun. Algebra} 25(8):2459--2506.

\bibitem{Lam1998} Lam, T. Y. (1998). Lectures on Modules and Rings. Graduate Texts in Mathematics. Vol. 189. Berlin: Springer-Verlag.





\bibitem{Lezamalatorre2017} Lezama, O., Latorre, E. (2017). Non-commutative algebraic geometry of semi-graded rings. {\em Int. J. Appl. Comput. Math.} {27}(4):361--389.

\bibitem{Lopes2023} Lopes, S. A. (2023). Noncommutative Algebra and Representation Theory: Symmetry, Structure \& Invariants. {\em Commun. Math.} 32(3):63--117.




\bibitem{Macdonald1973} Macdonald, I. G. (1973). Secondary representation of modules over a commutative ring. {\em Sympos. Math.} {11}:23--43.

\bibitem{Maksimov2000} Maksimov, V. M. (2000). On a generalization of the ring of skew Ore polynomials. {\em Russian  Math. Surveys} {55}(4):817--818.





\bibitem{Melkersson1998} Melkersson, L. (1998). Content and inverse polynomials on Artinian modules. {\em Commun. Algebra} {26}(4):1141--1145.

\bibitem{NinoRamirezReyes2020} Ni\~no, A., Ram\'irez, M. C., Reyes, A. (2020). Associated prime ideals over skew PBW extensions. {\em Commun. Algebra} {48}(12):5038--5055.

\bibitem{Ninoetal2024} Ni\~no, A., Ram\'irez, M. C., Reyes, A. (2024). A first approach to the Burchnall-Chaundy theory for quadratic algebras having PBW bases. \url{https://arxiv.org/abs/2401.10023}.

\bibitem{NinoReyes2023} Ni\~no, A., Reyes, A. (2023). On centralizers and pseudo-multidegree functions for non-commutative rings having PBW bases. {\em J. Algebra Appl.} 2550109.

\bibitem{Nordstrom2005} Nordstrom, H. E. (2005). Associated primes over Ore extensions and generalized Weyl algebras. PhD Thesis. University of Oregon.

\bibitem{Nordstrom2012} Nordstrom, H. E. (2012). Simple Modules Over Generalized Weyl Algebras and Their Associated Primes. {\em Commun. Algebra} {40}(9):3224--3235.




\bibitem{Ore1931} Ore, O. (1931). Linear Equations in Non-commutative Fields. {\em Ann. of Math. {\rm(}2{\rm)}} {32}(3):463--477.

\bibitem{Ore1933} Ore, O. (1933). Theory of Non-Commutative Polynomials. {\em Ann. of Math. {\rm(}2{\rm)}} {34}(3):480--508.

\bibitem{OuyangBirkenmeier2012} Ouyang, L., Birkenmeier, G. F. (2012). Weak annihilator over extension rings. {\em Bull. Malays. Math. Sci. Soc.} 35(2):345--347.

\bibitem{RamirezReyes2024} Ram\'irez, M. C., Reyes, A. (2024). A view toward homomorphisms and cv-polynomials between double Ore extensions. {\em Algebra Colloq.} To appear. \url{https://arxiv.org/abs/2401.14162}.




\bibitem{Smits1968} Smits, T. H. M. (1968). Skew polynomial rings. {\em Indag. Math. {\rm(}N.S.{\rm)}} {30}(1):209--224.

\bibitem{ZhangZhang2008} Zhang, J. J., Zhang, J. (2008). Double Ore extensions. {\em J. Pure Appl. Algebra} {212}(12):2668--2690.

\bibitem{ZhangZhang2009} Zhang, J. J., Zhang J. (2009). Double extension regular algebras of type (14641). {\em J. Algebra} 322(2):373--409.


\end{thebibliography}
\end{document}